\documentclass[12pt,reqno]{article}

\usepackage{amsmath,amsthm,amssymb,amsfonts}

\usepackage{graphicx}
\usepackage[dvipsnames]{xcolor}
\usepackage{verbatim}

\usepackage{authblk}

\usepackage{hyperref}

\theoremstyle{plain}
\newtheorem{theorem}{Theorem}
\newtheorem{corollary}[theorem]{Corollary}
\newtheorem{lemma}[theorem]{Lemma}
\newtheorem{proposition}[theorem]{Proposition}

\theoremstyle{definition}
\newtheorem{definition}{Definition}
\newtheorem{example}{Example}
\newcommand{\setEnvironmentQed}[2]{
\AtBeginEnvironment{#1}{%
    \pushQED{\qed}\renewcommand{\qedsymbol}{#2}%
  }
\AtEndEnvironment{#1}{\popQED}
}
\setEnvironmentQed{example}{\ensuremath{\blacksquare}}

\theoremstyle{remark}

\newtheorem*{remark*}{Remark}

\newcommand{\floor}[1]{\left\lfloor #1 \right\rfloor}

\begin{document}

\title{From a Voucher Puzzle to Extremal Sums of Adjacent Products}

\author{Chris Chen}
\author{Vivian Chen}
\author{Ray Cui}
\author{Ermin Dong}
\author{Alexander Radul}
\author{Lev Radul}
\author{Jack Shan}
\author{Arthur Shu}
\author{Kenneth Sun}
\author{Kenneth Wood}
\author{William Zelevinsky}
\author{Brian Zhao}
\affil{PRIMES STEP}
\author{Tanya Khovanova}
\affil{MIT}
\date{}

\maketitle

\begin{abstract}
Motivated by a self-referential puzzle, we study sequences of voucher price tags in which each choice multiplies the cost of the following one. We connect the puzzle setting to classical permutation statistics, introducing the \textit{voucher cost} alongside the related \textit{pairwise} and \textit{loop} costs. This perspective allows us to translate questions about budgeting into extremal problems on permutations. We review known results for permutations of ${1,2,\dots,n}$ and extend them to arbitrary sets of distinct non-negative price tags.
\end{abstract}

\section{Introduction}

This puzzle appears in the book \textit{Mathematical Puzzles and Curiosities} \cite{DKS}.

\begin{quote}
As you walk into the self-referential shop, you see a shelf full of vouchers. A voucher reading ``The next voucher you buy costs $N$ times its price tag'' has a price tag of $N$ dollars. Each voucher's multiplier applies only to the immediately following purchase. An infinite number of vouchers is in front of you, corresponding to $N=1,2,3,\ldots$. There is a single voucher available to purchase for each value of $N$. If, for example, you bought a voucher corresponding to $N=2$, followed by one corresponding to $N=3$, followed by one corresponding to $N=4$, you'd pay \$20 in total (\$2 for the first voucher, plus $\$3\cdot 2 = \$6$ for the second one, plus $\$4\cdot 3 = \$12$ for the third one).

Given a budget of 26 dollars, what is the maximum number of vouchers you can buy?
\end{quote}

Each voucher has a fixed \textit{price tag}, but purchasing a voucher also affects the cost of the next purchase by multiplying it by a specified factor. Even in the simplest setting, where the available vouchers have price tags $1,2,3,\dots$ and each multiplier applies only to the immediately following purchase, natural optimization questions arise. Given a fixed budget, how many vouchers can be bought? Equivalently, given a fixed number of vouchers, what is the minimum total cost, and how does this depend on the order in which the vouchers are purchased?

To study these questions systematically, we start Section~\ref{sec:prelim} with preliminaries and definitions. We introduce the \textit{voucher cost} of a sequence of vouchers, a quantity closely related to two well-known permutation statistics: the \textit{pairwise cost}, given by the sum of products of adjacent terms, and the \textit{loop cost}, which adds an additional product between the first and last terms. The pairwise and loop costs have appeared previously in various contexts, including extremal problems for permutations and applications to combinatorial optimization; see \cite{Klosinski, Mihai, Benoumhani}.

Our main focus is on extremal behavior.
For permutations of the natural numbers $\{1,2,\dots,n\}$, the maximum and minimum values of the pairwise and loop costs are known, and explicit formulas and constructions are available. We review these results and show how they fit naturally into the voucher framework.

In Section~\ref{sec:flippingtranscaling}, we then extend the analysis to arbitrary sets of distinct non-negative real price tags and introduce two key technical tools. The Flipping Lemma~\ref{lem:flipping} allows us to compare costs before and after reversing a certain subsequence, and it provides a unified explanation for many of the extremal constructions that arise. The Transcalation Lemma~\ref{lem:scaling} shows that a translation and scaling of a vector representing price tags preserve the relative ranking of loop costs.

In Section~\ref{sec:max}, we identify the permutations that maximize the loop, pairwise, and voucher costs and describe their structure, given that the price tags could be any set of distinct non-negative values. The voucher cost is more complicated to describe, and we define the \textit{min-to-min boundary-respecting construction} to help with that.

In Section~\ref{sec:min}, we identify the permutations that minimize the loop, pairwise, and voucher costs and describe their structure, given that the price tags can be any set of distinct non-negative values. We use the Flipping Lemma~\ref{lem:flipping}, the Transcalation Lemma~\ref{lem:scaling}, and the boundary-respecting constructions as the organizing tools. Our main contributions are to unify these three costs in a single framework, extend known extremal results from a set $\{1,\dots,n\}$ to arbitrary distinct non-negative price tags, and describe the maximizing and minimizing constructions in a common language.

Although the motivating problem is elementary to state, the resulting structure combines ideas from permutation statistics, extremal combinatorics, and greedy optimization. We hope that this paper highlights how a recreational puzzle can lead to a coherent and general mathematical theory.

In Section~\ref{sec:applications}, we discuss different applications of our results. In Section~\ref{sec:diffandranges}, we look at the achievable ranges and average values. In Sections~\ref{sec:arithmprog} and~\ref{sec:geometricprog}, we apply our formulas to integer arithmetic and geometric progressions. In Section~\ref{sec:neg}, we look at what happens if our price tags are negative. To do so, we generalize the voucher cost by allowing us to multiply the first voucher by a constant.

\section{Definitions and Preliminaries}
\label{sec:prelim}

\subsection{Definitions}

We denote the number of vouchers we buy by $n$, and the price tag of the $i$th voucher by $v_i$. The sequence of vouchers we buy becomes the vector $\vec{v}$:
\[v_1,\ v_2,\ v_3,\ \dots,\ v_n,\]
where the price tags are distinct non-negative real numbers. In the original puzzle, the price tags are integers $(1,2,3,\dots)$. We use parentheses for ordered sequences of price tags and braces for sets.

\begin{definition}
The total cost is denoted as $C_V(\vec{v})$ and is defined as
\begin{equation}
\label{cost}
    C_V(\vec{v}) = v_1 + v_1 v_2 + v_2 v_3 + v_3 v_4+ \cdots + v_{n-1} v_n,
\end{equation}
and we call it the \textit{voucher cost}. For voucher costs, we will also use the phantom coordinate $v_0 = 1$ to define the cost more homogeneously:
\[C_V(\vec{v}) = \sum_{i=1}^n v_{i-1} v_i.\]
\end{definition}

The original puzzle asks for the maximum number of vouchers under a fixed budget. We consider an equivalent question: the minimum cost, given the number of vouchers $n$. We denote this value as $m_V(n)$.

When buying $n$ vouchers and minimizing the voucher cost, it is natural to restrict ourselves to the first $n$ natural numbers. In this case, the sequence $\vec{v}$ becomes a permutation of length $n$. To emphasize that this is a permutation, we will sometimes denote it as $\pi$, and its elements as $p_i$.

The original puzzle did not ask about the maximum cost, as we cannot bound it, since the sequence of vouchers is infinite. When we restrict the $n$ purchased vouchers to be integers from $1$ to $n$, we can look at the maximum cost too, which we denote as $M_V(n)$.

In the original problem, the voucher prices form a sequence of natural numbers. Generalizing that, we sometimes assume that the sequence of voucher price tags is an increasing sequence of positive integers $s_i$.

For the minimum cost, as before, we need to buy the cheapest vouchers. When we are looking for the maximum cost, we also look at the maximum cost among the first cheapest vouchers. Thus, when we buy $n$ vouchers, we always buy among the first $n$ cheapest ones.

We are buying our $n$ vouchers $v_i$ in some order, which we want to define a relative order. We say that $\vec v$ is \textit{order-isomorphic} to $\pi=(p_1,\dots,p_n)$ if
\[
p_i < p_j \quad\Longleftrightarrow\quad v_i < v_j
\]
for all pairs $i,j$.

\begin{example}
The voucher sequences $(3,7,5)$ and $(2,8,4)$ are order-isomorphic to $\pi = (1,3,2)$.
\end{example}

The cost $C_V(\vec{v})$ is related to other costs that were studied before. Given the sequence of price tags, we introduce two additional cost types: the pairwise cost and the loop cost.

\begin{definition}
For a sequence of vouchers with at least two vouchers, we denote by $C_P(\vec{v})$ the \textit{pairwise cost}:
\[C_P(\vec{v}) = v_1 v_2 + v_2 v_3 + v_3 v_4+\cdots + v_{n-1} v_n.\] 
Similarly, we denote by $C_L(\vec{v})$ the \textit{loop cost}:
\[C_L(\vec{v}) = v_1 v_2 + v_2 v_3 + v_3 v_4+\cdots + v_{n-1} v_n + v_n v_1.\]
For $n=1$, these two costs are left undefined.
\end{definition}

The pairwise cost is similar to the voucher cost with the first term removed. The pairwise cost is the sum of pairwise products of consecutive elements of the given sequence of vouchers.

The loop cost is similar to the pairwise cost, but we assume that the numbers are on a loop, namely, we add the term $v_n v_1$ to the cost.

\begin{definition}
For the pairwise cost, we call two sequences \emph{equivalent} if they are reflections of each other. For the loop cost, we call two sequences \emph{equivalent} if one can be obtained from the other using reflections and rotations. For the voucher cost, two sequences are \emph{equivalent} if the subsequence of terms starting with $v_1$ and ending right before the $v_i = 1$ in one sequence is a reflection of the prefix of the other sequence, with the rest of the terms being the same. The last case is only relevant when one of the price tags is $1$.
\end{definition}

\begin{proposition}
For each cost definition, equivalent sequences have the same cost.
\end{proposition}

\begin{proof}
    The proof is by direct calculation.
\end{proof}

If the price tags are distinct, then for loop costs, when $n>2$, each equivalence class has $2n$ sequences. For pairwise costs, when $n>1$, each equivalence class has $2$ sequences. For the voucher cost when there is no price tag of $1$, the number is $1$. For the voucher cost, when one price tag is $1$ and $n>2$, we have two cases: a) a sequence with the first or second term equal to $1$ is only equivalent to itself; b) otherwise, it is equivalent to one other sequence.

\begin{example}
The voucher cost for the sequence $(3, 2, 1, 4)$ is the same as for the sequence $(2, 3, 1, 4)$.
\end{example}

\begin{example}
\label{ex:cv4}
    For $n=4$, we have three different equivalence classes for loop costs, and $8$ permutations for each cost. Given the sequence of natural numbers, permutations equivalent to $(1,2,4,3)$ have cost 25, those equivalent to $(1,2,3,4)$ have cost 24, and the ones equivalent to $(1,3,2,4)$ have cost 21.
\end{example}

Given two vectors $\vec{u}$ and $\vec{v}$ of lengths $\ell_1$ and $\ell_2$, respectively, we denote their concatenation as $(\vec{u},\vec{v})$, which has length $\ell_1 + \ell_2$.

The following proposition follows from an immediate calculation. It allows us to connect different costs.

\begin{proposition}
\label{prop:add0}
    We have
\[C_P(\vec{v}) = C_P((\vec{v},0)) = C_L((\vec{v},0)) = C_P((0,\vec{v})) = C_L((0,\vec{v})) = C_V((0,\vec{v})).\] 
    We also have 
    \[C_P((v_0,\vec{v},v_0)) = C_L((v_0,\vec{v})) \quad \textrm{ and } \quad C_V(\vec{v}) = C_P((1,\vec{v})) .\] 
\end{proposition}

\begin{proof}
The proof is done by plugging the values in, and we leave it to the reader.
\end{proof}

\subsection{Natural numbers}

\subsubsection{Maximum}

If the available price tags form a sequence of natural numbers, then the set of $n$ vouchers we buy is a permutation of $\{1,\dots,n\}$. In this case, the maximum and minimum possible values of the pairwise cost, given the length of the permutation, are known \cite{Klosinski, Mihai}. We denote these two values as $M_P(n)$ and $m_P(n)$, respectively. Similarly, the maximum and minimum values for the loop cost are known as well. We denote them as $M_L(n)$ and $m_L(n)$, respectively.

The maximum for the loop cost $M_L(n)$ is provided by the OEIS sequence A110610. The first few terms, starting with $n=2$, are:
\[4,\ 11,\ 25,\ 48,\ 82,\ 129,\ 191,\ 270,\ 368,\ 487,\ 629.\]

The formula for $M_L(n)$ is
\[M_L(n) = \frac{2n^3+3n^2-11n+18}{6}.\]

The maximum pairwise cost $M_P(n)$ is provided by the OEIS sequence A101986$(n-1)$. The first few terms, starting with $n=2$, are:
\[2,\ 9,\ 23,\ 46,\ 80,\ 127,\ 189,\ 268,\ 366,\ 485,\ 627.\]

Here is the procedure for finding all permutations that maximize the loop and pairwise costs. 

\textbf{Generating permutations for maximum loop and pairwise costs}: Put all the even numbers in increasing order, then put all the odd numbers in decreasing order.

\begin{example}
When $n=8$, the permutation realizing the maximum loop cost of 191 is
\[2,\ 4,\ 6,\ 8,\ 7,\ 5,\ 3,\ 1\]
and its rotations and reflections. The same permutation and its reflection realize the maximum pairwise cost of 189.
\end{example}

The following proposition is known \cite{Benoumhani, OEIS}.

\begin{proposition}
The construction above provides the maximum for both the loop and the pairwise cost. The maximum is achieved only for the permutations equivalent to the one described above. Also,
\[M_P(n) = M_L(n) - 2.\]
\end{proposition}

\subsubsection{Minimum}

The minimum for the loop cost $m_L(n)$ is provided by the OEIS sequence A110611. The first few terms, starting with $n=2$, are
\[4,\ 11,\ 21,\ 37,\ 58,\ 87,\ 123,\ 169,\ 224,\ 291,\ 369,\ 461.\]

The formula for $m_L(n)$ is the following
\begin{equation}
\label{eq:loopmin}
m_L(n) = \frac{n^3+3n^2+5n}{6} - \frac{9 + 3(-1)^n}{12}.
\end{equation}

We can write the same formula using cases:
\[m_L(n)=
\begin{cases}
\frac{n^3 + 3n^2 + 5n - 6}{6} & \text{if } n \text{ is even} \\
\frac{n^3 + 3n^2 + 5n - 3}{6} & \text{if } n \text{ is odd.}
\end{cases}\]

The minimum value of the pairwise cost $m_P(n)$ is given by sequence A026035 in the OEIS \cite{OEIS} database. 
Keep in mind that this sequence starts from index 2, and the first few terms are
\[2,\ 5,\ 12,\ 22,\ 38,\ 59,\ 88,\ 124,\ 170,\ 225,\ 292,\ 370,\ 462,\ 567.\]

The formula for $m_P(n)$ is the following:
\[m_P(n) = \frac{2n^3 + 4n - 3 + 3(-1)^n}{12} = \binom{n}{3} + \floor{\frac{n^2}{2}}.\]

To parallel the maximum case, we look at the difference. The formula for the difference between the minimum loop and pairwise cost is as follows:
\[m_L(n) - m_P(n) = \frac{n^2 + n - 1 - (-1)^n}{2}=\begin{cases}
\frac{n(n+1)}{2} & \text{if } n \text{ is odd} \\
\frac{(n+2)(n-1)}{2} & \text{if } n \text{ is even}
\end{cases}.\]

\textbf{Generating permutations for minimum loop cost}: Put the largest number first and the smallest number last. Then put the second smallest number second, and the second largest number second-to-last. Repeat the process to get
\[n,\ 2,\ n-2,\ 4,\ \dots,\ n-3,\ 3,\ n-1,\ 1.\]
Together with its rotations and reflections, this gives all minimizing permutations.

\begin{example}
When $n=8$, the permutations realizing the minimum loop cost of 123 are
\[8,\ 2,\ 6,\ 4,\ 5,\ 3,\ 7,\ 1\]
and its rotations and reflections.
\end{example}

\textbf{Generating permutations for minimum pairwise cost}: Put the odd numbers in decreasing order, and the even numbers in increasing order. Then, interlace these two sequences starting with the largest odd number. If $n$ is odd, we get the following sequence:
\[n,\ 2,\ n-2,\ 4,\ \dots,\ n-3,\ 3,\ n-1,\ 1.\]
which is the same as our sequence above for the minimal loop cost. For even $n$, we get the following sequence:
\[n-1,\ 2,\ n-3,\ 4,\ \dots,\ n-4,\ 3,\ n-2,\ 1,\ n.\]
Its reverse also attains the same value.

\begin{example}
When $n=8$, the permutations realizing the minimum pairwise cost of 88 are
\[7,\ 2,\ 5,\ 4,\ 3,\ 6,\ 1,\ 8\]
and
\[8,\ 1,\ 6,\ 3,\ 4,\ 5,\ 2,\ 7.\] 
\end{example}

\subsubsection{Maximum minus minimum}

The difference $M_L(n) - m_L(n)$ is sequence A306262 in the OEIS database \cite{OEIS} shifted by 1: $M_L(n) - m_L(n) = \textrm{A306262}(n-1)$. The first few terms of $M_L(n) - m_L(n)$ starting from index 2 are:
\[0,\ 0,\ 4,\ 11,\ 24,\ 42,\ 68,\ 101,\ 144,\ 196,\ 260.\]

The formula for the difference between the maximum loop and minimum loop cost is as follows:
\[M_L(n) - m_L(n) = \frac{2n^3 -32n+45+3(-1)^n}{12}=\begin{cases}
\frac{2n^3 -32n+42}{12} & \text{if } n \text{ is odd} \\
\frac{2n^3 -32n+48}{12} & \text{if } n \text{ is even}
\end{cases}.\]

The difference $M_P(n) - m_P(n)$ is provided by the same sequence A306262, but now it is not shifted. Here are a few terms starting from index 2:
\[0,\ 4,\ 11,\ 24,\ 42,\ 68,\ 101,\ 144,\ 196,\ 260,\ 335,\ 424.\]
We have
\[M_P(n) - m_P(n) = \frac{2n^3+6n^2-26n+15-3(-1)^n}{12}.\]

\section{The Flipping Lemma}
\label{sec:flippingtranscaling}

Recall that when we calculate the voucher cost, we assume that $v_0 = 1$.

We state the following lemma in a more general form. It applies to arbitrary real price tags, not necessarily distinct.

\begin{lemma}[The Flipping Lemma]
\label{lem:flipping}
    Suppose we have a sequence of price tags $v_1,v_2,\dots,v_n$, where the price tags are allowed to be any real numbers, not necessarily distinct. If $i < j - 1$ and
    \[(v_i-v_j)(v_{i+1}-v_{j-1}) < 0,\]
    if we reverse the subsequence starting from $v_{i+1}$ through $v_{j-1}$, then the new loop, pairwise, and voucher costs are greater than the original ones. Otherwise, if $(v_i-v_j)(v_{i+1}-v_{j-1}) > 0$, the costs are smaller, and if $(v_i-v_j)(v_{i+1}-v_{j-1}) = 0$, the costs remain unchanged.
\end{lemma}

\begin{proof}
Note that after reversing the subsequence starting from $v_{i+1}$ through $v_{j-1}$, the new sequence is
\[v_1,\ v_2,\ \ldots,\ v_{i-1},\ v_i,\ v_{j-1},\ v_{j-2},\ \ldots,\ v_{i+1},\ v_j,\ v_{j+1},\ \ldots,\ v_n.\]
Subtracting the new cost value from the original loop cost (similarly for pairwise and voucher costs), we see that the difference affects two terms only, and we get that the difference is
\[v_iv_{i+1} + v_{j-1}v_j - v_iv_{j-1} - v_{i+1}v_j = (v_i-v_j)(v_{i+1}-v_{j-1}),\]
implying the stated result.
\end{proof}

We denote a constant vector of length $n$, where each coordinate is $c$ by $\vec{c}$. Given vector $\vec{v} = (v_1,v_2,\ldots,v_n)$, the vector $a\vec{v} + \vec{c}$ equals $(av_1+c,av_2+c,\ldots,av_n+c)$.

The following lemma describes how the relative orders of loop costs of two permutations of $\vec{v}$ and $a\vec{v} + \vec{c}$ are connected. But first, we introduce a notion of rank.

Suppose the loop cost (similarly, pairwise and voucher costs) achieves a set $A$ of costs on permutations of $n$ elements. Suppose we sort all possible costs in order $A = \{a_1, a_2, \dots, a_k\}$, where $a_i$ are distinct and listed in increasing order. We say that the \textit{rank} of cost $a_i$ is $i$. For example, the rank of the minimum cost is 1.

The following lemma explains that the ranking order does not change with translations and scaling.

\begin{lemma}[Transcalation Lemma]
\label{lem:scaling}
    The loop costs of vectors $\vec{v}$ and $a\vec{v} + \vec{c}$ have the same ranks among all permutations of vectors $\vec{v}$ and $a\vec{v} + \vec{c}$, respectively, where we assume that $a \ne 0$.
\end{lemma}

\begin{proof}
    To compare the costs, it is enough to compare each term. We have
    \[(av_i+c)(av_{i+1} + c) = a^2v_iv_{i+1} + ac(v_i + v_{i+1}) +c^2.\]
    Hence
    \[C_L(a\vec{v} + \vec{c}) = a^2C_L(\vec{v}) + 2ac\left(\sum_{i=1}^n v_i\right) +nc^2.\]
    Thus, the transition from cost to cost is a linear function with a positive first coefficient. Therefore, this transition is a bijection on the sets of costs preserving the order. The statement follows.
\end{proof}

\begin{corollary}
    Suppose $a \ne 0$, and a permutation $\vec{v}$ of a sequence of price tags $s$ provides the minimum/maximum loop cost. Then the permutation $a\vec{v} + \vec{c}$ provides the minimum/maximum loop cost for a sequence of price tags $(as_1+c,\dots,as_n+c)$.
\end{corollary}

Adding a constant allows us to connect the loop cost for a set of $n$ vouchers, with the last voucher having the smallest price, to the pairwise cost for a set of $n-1$ vouchers.

\begin{corollary}
    If a permutation $v_1,v_2,\ldots,v_{n-1},s_1$ of an increasing sequence $s_1,s_2,\ldots,s_n$ provides the loop cost of rank $r$ among all vectors with the last value $s_1$, then the permutation
    \[v_1 - s_1, v_2 - s_1, \ldots, v_{n-1} - s_1\]
    provides the pairwise cost of the same rank among all vectors with values $v_1 - s_1, v_2 - s_1, \ldots, v_{n-1} - s_1$.
\end{corollary}

\begin{proof}
    Consider a loop cost for vector $v_1,v_2,\ldots,v_{n-1},s_1$, and subtract $s_1$ from each value. We get a vector
    \[v_1-s_1,v_2-s_1,\ldots,v_{n-1}-s_1,0.\]
    The loop cost of this vector is the same as the pairwise cost of the same vector with the last 0 removed. By the Transcalation Lemma~\ref{lem:scaling}, this is an order-preserving bijection between the sets of loop costs for permutations of $\vec{v}$ ending in $s_1$ and loop costs of permutations of the vector $(v_1-s_1,v_2-s_1,\ldots,v_{n-1}-s_1,0)$. The latter costs are in bijection with pairwise costs of permutations of $v_1-s_1,v_2-s_1,\ldots,v_{n-1}-s_1$. Thus, the rank is preserved.
\end{proof}

\begin{corollary}
\label{cor:FromLoopToPairwise}
    To find the order that maximizes or minimizes the pairwise cost for $n$ vouchers, it suffices to find the corresponding loop cost for $n+1$ vouchers, rotate it to make $1$ be the last value, remove the last value, and subtract $1$ from all remaining values. 
\end{corollary}

\section{Maximal Cost}
\label{sec:max}

Let $\vec{v}$ have distinct coordinates. Let $s=(s_1,s_2,\dots,s_{n-1},s_n)$ be the sequence of the same vouchers in order from least to greatest.

In the next theorem, the vouchers' price tags can be any set of distinct non-negative real numbers.

\begin{theorem}
\label{thm:maxloop}
For a set of price tags $s$ of distinct non-negative real numbers, a permutation of $s$ has maximal loop cost if and only if it is equivalent to the permutation of $s$ that is order-isomorphic to
\[2,\ 4,\ 6,\ 8,\ \ldots,\ n,\ \ldots,\ 7,\ 5,\ 3,\ 1.\]
The same order provides the maximal pairwise cost. If $s_1 > 0$, the above order providing maximal cost is unique up to the equivalence. If $s_1 = 0$, the maximal pairwise cost is also achieved for permutations that are order-isomorphic to
\[1,\ 2,\ 4,\ 6,\ 8,\ \ldots,\ n,\ \ldots,\ 7,\ 5,\ 3\]
and the reverse of it.
\end{theorem}

\begin{proof}
Assume we have $\vec{v}$ that maximizes the sum of products of adjacent values. We also assume that the indices of $\vec{v}$ wrap around.

Assume that for some $s_k < s_n$, the set of terms $s_1$ through $s_k$ forms a single connected chain of adjacent values, but $s_{k+1}$ is not adjacent to any term in that chain. Then there exists a connected chain of terms that contains all terms from $s_1$ through $s_k$, followed by terms greater than $s_{k+1}$, then followed by $s_{k+1}$, then followed by terms greater than $s_{k+1}$. Suppose $i$ is the last index in the chain of terms from $s_1$ through $s_k$ and $j-1$ is the index of $s_{k+1}$. Then $v_i < s_{k+1} < v_j$ and $v_{i+1} > v_{j-1} = s_{k+1}$. Thus, we can apply the Flipping Lemma~\ref{lem:flipping} and increase the cost.

Thus, $s_{k+1}$ must be adjacent to the string with terms $s_1$ through $s_k$. Suppose it is adjacent to the term $a$, while the term $b$ is on the other side of the chain. Now $b$ is adjacent to term $d > s_{k+1}$ on the other side of the chain. If we cut out the string with terms $s_1$ through $s_k$ and reverse the order, the new cost increases by
\[s_{k+1}b + ad - s_{k+1}a - bd = (s_{k+1} - d)(b-a).\]
It follows that $s_{k+1}$ must be adjacent to the smallest end term of the chain $s_1$ through $s_k$.

We can build the resulting sequence by induction on $k$. Thus, our sequence is equivalent to
\[\ldots,\ s_7,\ s_5,\ s_3,\ s_1,\ s_2,\ s_4,\ s_6,\ \ldots,\]
which, after rotation, is order-isomorphic to the string in the statement.

The pairwise cost can be derived from the loop cost as follows:
\[C_P(\vec{v}) = C_L(\vec{v}) - v_1v_n.\]
The maximum of this expression is achieved when $C_L(\vec{v})$ is the largest and $v_1v_n$ is the smallest. If $s_1 > 0$, the latter is the smallest when it equals $s_1s_2$, concluding the proof in this case.

If $s_1 = 0$, then $s_1s_j$ is one of the smallest pairwise products possible. This implies that $v_n = s_1$. After that, we can remove $s_1$ and apply the results above to the remaining values. Equivalently, the remaining values should be arranged as order-isomorphic to 
\[2,\ 4,\ 6,\ 8,\ \ldots,\ n-1,\ \ldots,\ 7,\ 5,\ 3,\ 1\]
or its reversal. Then we can add $s_1$ on either side.
\end{proof}

\begin{remark*}
    Here is an alternative way to prove the theorem above for the pairwise cost. We know that the maximum loop cost for $n+1$ vouchers is realized when the permutation is order-isomorphic to 
    \[2,\ 4,\ 6,\ 8,\ \dots,\ n+1,\ \dots,\ 7,\ 5,\ 3,\ 1.\]
    By Corollary~\ref{cor:FromLoopToPairwise}, the maximum pairwise cost for a permutation of length $n$ is realized for a permutation order-isomorphic to
    \[1,\ 3,\ 5,\ 7,\ \dots,\ n,\ \dots,\ 6,\ 4,\ 2,\]
    which, after a reflection, produces the same result as in Theorem~\ref{thm:maxloop}.
\end{remark*}

For the sequence of price tags consisting of the natural numbers, the maximal loop and pairwise costs are known and shown in Section~\ref{sec:prelim}.

\begin{example}
    For $n=7$ and the sequence of natural numbers as price tags, the maximal loop and pairwise costs are realized for the sequence $(2, 4, 6, 7, 5, 3, 1)$ with the loop cost of $129$ and the pairwise cost of $127$.
\end{example}

Given an infinite increasing sequence $S$ of voucher price tags, we denote the maximum loop, pairwise, and voucher costs for buying $n$ vouchers as $M_L^S(n)$, $M_P^S(n)$, and $M_V^S(n)$. Similarly, for minimum costs, the notation is $m_L^S(n)$, $m_P^S(n)$, and $m_V^S(n)$.

From the theorem above, we have the following corollary.

\begin{corollary}
    We have
    \[M_L^S(n) = M_P^S(n) + s_1s_2.\]
\end{corollary}

\begin{example}
    We already know from Section~\ref{sec:prelim} that for the sequence of natural numbers, we have
    \[M_L(n) = M_P(n) + 2.\]
\end{example}

We now move to the voucher cost. Before proving the theorem, we define the min-to-min boundary-respecting construction. We initialize the value of the left boundary to $1$ and the value of the right boundary to $0$. 

\begin{definition}[Min-to-min boundary-respecting construction] The \textit{min-to-min boundary-respecting construction} works as follows. Pick a voucher from the vouchers that are not placed yet with the smallest price tag $x$ and put it on the side of the smallest current boundary value next to the other placed vouchers and on the inside. Change the corresponding boundary value to $x$. If there is a tie, we can choose either side.
\end{definition}

\begin{example}
    Suppose we have vouchers with price tags $(0.5, 2, 3, 4)$. Initially, the left boundary is $1$, and the right boundary is $0$, so the right boundary is smaller. The cheapest voucher, $0.5$, goes to the right, and the right boundary changes value to $0.5$. The right boundary is still the smallest, so the next cheapest voucher, $2$, goes to the right, before the $0.5$-voucher. Now the left boundary is $1$ and cheaper than the new value of the right boundary, which is $2$. Thus, the next voucher, $3$, goes to the left. The last voucher is placed in the remaining spot in the middle, giving us the following sequence of vouchers: $(3, 4, 2, 0.5)$.
\end{example}

\begin{theorem}
\label{thm:maxorderiso}
    Given a set of positive, distinct values, the following order yields the maximum voucher cost. First, all positive real numbers greater than or equal to one are arranged to be order-isomorphic to \[2,\ 4,\ 6,\ 8,\ \ldots,\ n,\ \ldots,\ 7,\ 5,\ 3,\ 1,\]
    and then all positive real numbers less than $1$ are placed at the end in decreasing order.
\end{theorem}

\begin{proof}
We first show that we can get the highest cost by following the min-to-min boundary-respecting construction. We prove it by induction on the number of placed vouchers.

For our base case, we assume that nothing has been placed yet. Now for the inductive step. Suppose we already placed the vouchers with values $s_{k-1}$ and smaller to maximize the voucher cost. Denote the left-side boundary value as $v_\ell$ and the right-side boundary value as $v_r$. Without loss of generality, suppose $v_\ell < v_r$. We now need to prove that the maximum cost is realized when the voucher $s_k$ is placed inside and to the right of $v_\ell$. 

Suppose, to the contrary, the voucher cost is maximized when the value $v_m > s_k$ is placed next to the smallest boundary value. Now we use the Flipping Lemma~\ref{lem:flipping}, where we flip the subsequence between the left boundary and $s_k$. Namely, we pick $v_i = v_\ell$ and $v_{j-1} = s_k$. We know that $v_j$ is either $v_r$ or greater, in any case $v_j > v_i = v_\ell$. In addition, $v_{i+1} = v_m > v_{j-1} = s_k$, by our assumption. We get $(v_i - v_j)(v_{i+1} - v_{j-1}) < 0$, and, by the Flipping Lemma~\ref{lem:flipping}, the described flip increases the cost. Thus, $s_k$ must be placed next to $v_\ell$ to the right, concluding the induction step.

The theorem follows from the use of the min-to-min boundary-respecting construction.
\end{proof}

If all our real numbers are greater than or equal to $1$, then we have the following corollary from the theorem.
\begin{corollary}
\label{cor:maxorderisoMoreThan1} 
For a sequence $s$ of distinct real numbers in increasing order with $s_1 \ge 1$, the permutation with maximal voucher cost is order-isomorphic to
\[2,\ 4,\ 6,\ 8,\ \dots,\ n,\ \dots,\ 7,\ 5,\ 3,\ 1.\]
In addition, if $s_1=1$, the maximal voucher cost is also achieved on permutations order-isomorphic to
\[3,\ 5,\ 7,\ 9,\ \dots,\ n,\ \dots,\ 6,\ 4,\ 2,\ 1.\]
\end{corollary}

\begin{proof}
The first part follows directly from Theorem~\ref{thm:maxorderiso}. Now suppose $s_1=1$. In this case, given that $v_n = 1$, the voucher cost is the same as the loop cost, and we have two loops that achieve the maximum and end in $s_1$: the one we had before and the one that is order-isomorphic to
\[3,\ 5,\ 7,\ 9,\ \dots,\ n,\ \dots,\ 6,\ 4,\ 2,\ 1.\]

Or, if we follow the min-to-min boundary-respecting construction, after placing the smallest voucher to the right, the two boundary values become equal, and we can place the second-smallest voucher either to the right or to the left.
\end{proof}

For the set of vouchers that are integers 1 through $n$, we can connect the maximal voucher cost to the maximal loop and pairwise costs.

\begin{corollary}
    We have
    \[M_V^S(n) = M_P^S(n) + s_2 = M_L^S(n) - s_1s_2 + s_2.\]
\end{corollary}

In particular, for our main example, we have the following formulas.

\begin{corollary}
    We have
    \[M_V(n) = M_P(n) + 2 = M_L(n) = \frac{2n^3+3n^2-11n+18}{6}.\]
\end{corollary}

We can also describe the sequence maximizing the pairwise cost using our boundary construction.

\begin{corollary}
    We can use the min-to-min boundary-respecting construction, with initial values 0 on both sides, to construct a sequence with maximum pairwise cost.
\end{corollary}

\begin{proof}
    We use the formula $C_P(\vec{v}) = C_P((0,\vec{v})) = C_V((0,\vec{v}))$ from  Proposition~\ref{prop:add0}. The pairwise cost is equivalent to multiplying the first value by $0$ at the beginning of the voucher cost, rather than by $1$, so we initialize the left boundary to $0$ rather than $1$. From there, we can follow the min-to-min boundary procedure since the pairwise cost is essentially the same as the voucher cost.
\end{proof}

We can also describe the sequence maximizing the loop cost using our boundary construction.

\begin{corollary}
    We can use the min-to-min boundary-respecting construction, with initial values $s_1$ on both sides, and proceed with the remaining values to construct a sequence with maximum loop cost.
\end{corollary}

\begin{proof}
We can represent our vector $\vec{v}$ as a concatenation $(s_1,\vec{w})$. From Proposition~\ref{prop:add0}, we have $C_P((s_1,\vec{w},s_1)) = C_L((s_1,\vec{w})) = C_L(\vec{v})$. If we follow the boundary procedure for the pairwise cost on vector $(s_1,\vec{w},s_1)$, we will put two copies of $s_1$ as new boundaries on the left and on the right, and then continue. This is equivalent to initializing the boundaries as $s_1$ and proceeding with the remaining values.
Once again, rotations and reflections of the resulting sequence work as well.
\end{proof}

\begin{example}
    Suppose we have vouchers with price tags $(0.5, 1, 4, 5)$. We start by placing the $0.5$ on the left and initializing both boundaries to $0.5$. Now, we move on to the $1$, which we can place on either side. If we place it on the left, then the $4$ goes next to the $0.5$ boundary on the right, and the $5$ goes next to the $1$. This gives the sequence $(0.5,1,5,4)$. If we place the $1$ on the right, then the $4$ goes next to the $0.5$ on the left, and the $5$ goes next to the $1$. This gives the equivalent sequence $(0.5,4,5,1)$.
\end{example}

\section{Minimal Cost}
\label{sec:min}

\subsection{Preliminary lemmas}

We start with the following two lemmas, which apply to any set of distinct non-negative price tags. As before, we assume that the ordered version of $\vec{v} = (v_1,v_2,\ldots,v_n)$ is the increasing sequence $s = (s_1,\ldots,s_n)$.

\begin{lemma}
\label{lem:pnisn}
    Given distinct non-negative price tags, there exists a sequence that provides the minimum loop, pairwise, or voucher cost and is equivalent to a sequence
$\vec v=(v_1,\ldots,v_n)$ with $v_n=s_n$.
\end{lemma}

\begin{proof}
For the loop cost, every sequence is equivalent to a sequence where the last term is $s_n$.

Now consider the pairwise cost. Assume for the sake of contradiction that $v_n\neq s_n$. If $v_1=s_n$, then we can reflect the sequence to get an equivalent sequence with $v_n=s_n$. Suppose $v_1 \neq s_n$. Let $v_i$ be the term before $s_n$, then $i > 0$. Let $v_{n+1}=0$. Now we apply Lemma~\ref{lem:flipping} with $v_i$ the term before $s_n$ and $j = n+1$. We have $v_i \ge 0=v_{n+1}=v_j$ and $v_{i+1}=s_n>v_n=v_{j-1}$. This means that after flipping our subsequence, the new sequence has a cost no greater than the current cost and has $v_n=s_n$. Thus, we can assume that $v_n = s_n$. If $v_i > 0$, then, after flipping, the cost decreases, creating a contradiction.

Now consider the voucher cost. Let $v_i$ be the term before $s_n$. Recall that we set $v_0=1$, which makes $v_i$ always defined, and also, the voucher cost can be expressed through the pairwise cost of vector $(1,\vec{v})$. Let $v_{n+1}=0$. Now we can proceed in the same way as in the pairwise cost proof above.
\end{proof}

\begin{lemma}
\label{lem:pn-1is1}
If $n>1$, then there exists a sequence $\vec{v}$ providing the minimal loop, respectively pairwise, cost with $v_{n-1} = s_1$ and $v_n = s_n$.
\end{lemma}

\begin{proof}
By Lemma~\ref{lem:pnisn}, we know that there exists a vector $\vec{v}$ with the minimal cost such that $v_n = s_n$. Suppose $v_{n-1} \ne s_1$ and $v_h = s_1$, where $h < n-1$. 

For the loop cost, if $v_1=s_1$ then we can reflect and rotate the sequence to get an equivalent sequence with $v_n=s_n$ and $v_{n-1} = s_1$.

Let $v_i$ be the term before $s_1$. For the loop cost, the term $s_1$ is not the first, so $i > 0$. For the pairwise cost, we let $v_0 = 0$, which ensures $v_i$ is always defined.

Now we apply the Flipping Lemma~\ref{lem:flipping} with $v_i$ as defined above and $j = n$. We have $v_i < v_j = v_n = s_n$ and $v_{i+1} = s_1 < v_{n-1} = v_{j-1}$. This means the new sequence has a smaller cost, which is a contradiction. Thus, $v_n=s_n$ and $v_{n-1} = s_1$.
\end{proof}

We now show that for the minimal loop cost, the maximum value is surrounded by the two smallest values.

\begin{lemma}
\label{lem:s2}
    For $n \ge 3$, the minimal loop cost is achieved when $v_1=s_{2}$, $v_{n-1} = s_1$, and $v_n = s_n$.
\end{lemma}

\begin{proof}
By Lemma~\ref{lem:pn-1is1}, we know that any sequence that provides the minimal loop cost is equivalent to a sequence $\vec{v}$ that satisfies $v_{n-1}=s_1,$ and $v_n=s_n$.

Let $s_2$ be in position $\ell \ne 1$. Then, since we look at a loop, we can assume that $v_0= v_n = s_n$. We apply the Flipping Lemma~\ref{lem:flipping}, where $i = 0$ and $j = \ell + 1$. The change in cost if we reverse the subsequence with indices $1$ to $\ell$ is $(v_0-v_{\ell+1})(v_1 - v_\ell) = (s_n-v_{\ell+1})(v_1 - s_2)$. Since $v_1$ is neither $s_1$ nor $s_2$, we have $v_1>s_2$; also $s_n>v_{\ell+1}$. Thus, this difference is positive, so the flip decreases the cost, a contradiction.
\end{proof}

\begin{corollary}
\label{cor:2smthn-11n}
    For $n \ge 4$, the minimal loop cost is achieved, when $v_1=s_{2}$, $v_{n-2} = s_{n-1}$, $v_{n-1} = s_1$, and $v_n = s_n$. 
\end{corollary}

\begin{proof}
    Suppose we subtract a number $N$ larger than $v_n$ from every coordinate of vector $\vec{v}$. We get a vector with all negative coordinates $(v_i - N)$. Changing the signs in every coordinate of the new vector, we get a vector with all positive coordinates $(N - v_i)$, which achieves the minimum when its maximum value $N-s_1$ is surrounded by the two smallest values $N-s_n$ and $N - s_{n-1}$. By the Transcalation Lemma~\ref{lem:scaling}, in the original permutation that achieves the minimum, the smallest value has to be surrounded by the two largest values.

    It follows that after we choose a representative with the minimal loop cost with $v_1=s_{2}$, $v_{n-1} = s_1$, and $v_n = s_n$, which is possible by Lemma~\ref{lem:s2}, then $v_{n-2} = s_{n-1}$.
\end{proof}

\subsection{Minimal loop cost}

We calculated the pattern for the minimal loop cost using a program and verified that it follows the order pattern in the following theorem for the number of vouchers up to 12.

\begin{example}
\label{ex:min4-5}
    The minimum cost for vouchers in the set $\{1,2,3,4\}$ is achieved by the sequence $(2,3,1,4)$. The minimum cost for the set $\{1,2,3,4,5\}$ is achieved by the sequence $(2,3,4,1,5)$. 
\end{example}

\begin{theorem}
\label{thm:loopmin}
A permutation $v_1,v_2,\ldots,v_n$ of an increasing non-negative sequence $s$ provides the minimum loop cost if and only if it is equivalent to a sequence that is order-isomorphic to
\[2,\ n-2,\ 4,\ \dots,\ n-3,\ 3,\ n-1,\ 1,\ n.\]
\end{theorem}

\begin{proof}
    We prove this by induction. The base cases are provided by direct calculation, see Example~\ref{ex:min4-5}.

    Assume that the statement is true for $n$ vouchers. We now consider $n+2$ vouchers. By Corollary~\ref{cor:2smthn-11n}, the minimal loop cost is achieved on a sequence that is order-isomorphic to 
    \[2,\ \dots,\ n+1,\ 1,\ n+2,\]
    where the dots represent the order that is not proven yet. If we remove the vouchers $s_1$ and $s_{n+2}$, the cheapest and most expensive among the remaining vouchers are adjacent on the loop. That means, we can achieve the minimum if we arrange the remaining vouchers according to our induction hypothesis.

    Thus, we consider the sequence of $n$ price tags consisting of $s_2$, $s_3$, $\ldots$, $s_n$, $s_{n+1}$. By the induction hypothesis, if we make the sequence order-isomorphic to
    \[2,\ n-2,\ 4,\ \dots,\ n-3,\ 3,\ n-1,\ 1,\ n,\]
    the loop cost is minimized. It follows that the vouchers are in order 
    \[s_3,\ s_{n-1},\ s_5,\ \dots,\ s_{n-2},\ s_4,\ s_n,\ s_2,\ s_{n+1}.\]
    We now shift the last voucher to the left and reverse the order to get
    \[s_2,\ s_{n},\ s_4,\ s_{n-2},\ \dots,\ s_5,\ s_{n-1},\ s_3,\ s_{n+1}.\]
    Attaching the two vouchers $s_1$ and $s_{n+2}$ at the end, we obtain the order needed to complete the induction.
\end{proof}

\begin{example}
    We know that one of the sequences providing the minimal loop cost is order-isomorphic to
    \[2,\ n-2,\ 4,\ \dots,\ n-3,\ 3,\ n-1,\ 1,\ n.\]
    If we shift this to the right by 1, we get a sequence with the same loop cost that is order-isomorphic to 
    \[n,\ 2,\ n-2,\ 4,\ \dots,\ n-3,\ 3,\ n-1,\ 1.\]
    For the sequence of natural numbers, this is the same minimal loop cost we described in Section~\ref{sec:prelim}.
\end{example}

\subsection{Minimal Pairwise cost}

We can derive the minimum pairwise cost from the minimum loop cost.

\begin{theorem}
\label{thm:minpairwise}
    A permutation $\vec{v}$ of distinct non-negative price tags provides the minimum pairwise cost if and only if it is equivalent to a sequence that is order-isomorphic to
\[n-1,\ 2,\ n-3,\ \dots,\ n-4,\ 3,\ n-2,\ 1,\ n.\]
\end{theorem}

\begin{proof}
By Corollary~\ref{cor:FromLoopToPairwise}, we consider the order of a sequence of $n+1$ vouchers minimizing the loop cost with the smallest voucher at the end. From Theorem~\ref{thm:loopmin}, we know that it is order-isomorphic to
\[n+1,\ 2,\ n-1,\ 4,\ \dots,\ n-2,\ 3,\ n,\ 1.\]
Using Corollary~\ref{cor:FromLoopToPairwise}, one of the sequences of vouchers realizing the minimum pairwise cost must be order-isomorphic to
\[n,\ 1,\ n-2,\ \dots,\ n-3,\ 2,\ n-1,\]
which, after a reflection, is the same as in the statement, concluding the proof.
\end{proof}

We have the following connection between the minimum pairwise cost for a sequence $S$ of $n$ vouchers and the minimum loop cost for a sequence $S'$ of $n-1$ vouchers, in which the last voucher from $S$ is removed.

\begin{proposition}
    We have \[m_P^S(n) = m_L^{S'}(n-1) + s_1(s_n-s_{n-1}).\]
\end{proposition}

\begin{proof}
The sequence realizing the minimum loop cost of $n-1$ vouchers is order-isomorphic to 

\[2,\ n-3,\ 4,\ \dots,\ n-4,\ 3,\ n-2,\ 1,\ n-1.\]
Shifting $n-1$ to the front, we get
\[n-1,\ 2,\ n-3, \dots,\ 3,\ n-2,\ 1.\]
By Theorem~\ref{thm:minpairwise}, if we add $n$ at the end of this sequence, we get a sequence realizing the minimum pairwise cost for $n$ vouchers. It follows that for a sequence $S$ of $n$ terms and a subsequence $S'$ of $S$ with the last term removed, we have
\[m_P^S(n) = m_L^{S'}(n-1) + s_1(s_n-s_{n-1}).\]
\end{proof}

\begin{example}
\label{ex:mpn-mLn-1+1}
    In particular,
    \[m_P(n) = m_L(n-1) + 1.\]
\end{example}

\subsection{Minimal Voucher cost}

Before proving the theorem, we define the alternating max-to-min and min-to-max boundary-respecting construction. We initialize the value of the left boundary to $1$ and the value of the right boundary to $0$.

\begin{definition}[Alternating max-to-min and min-to-max boundary-respecting construction] 
The \textit{alternating max-to-min and min-to-max boundary-respecting construction} works as follows. We start by picking a voucher with the greatest price tag $x$ and put it on the right, replacing the value for the right boundary with $x$. Then pick a voucher with the smallest price tag that has not been placed, $y$, and place it next to the largest current boundary value, next to the other placed vouchers, and on the inside. Change the corresponding boundary value to $y$. If there is a tie, we can choose either side. Continue alternating: place the smallest remaining voucher next to the larger boundary, then the largest remaining voucher next to the smaller boundary.
\end{definition}

\begin{example}
Suppose we have price tags $0.2$, $2$, $4$, and $5$. Initially, the left boundary is $1$ and the right boundary is $0$, so the right boundary is smaller. We first place $5$, the largest value not yet placed, on the right, and the right boundary changes to $5$. We then place the $0.2$, the smallest value not yet placed, which goes to the right since $5>1$. The right boundary changes to $0.2$. Now, we place $4$ next to $0.2$, and the right boundary changes to $4$. Finally, we place $2$ next to $4$, and the right boundary becomes $2$. So the sequence that minimizes the voucher cost is $(2, 4, 0.2, 5)$.
\end{example}

\begin{theorem}
    If we follow the alternating max-to-min and min-to-max boun\-dary-respecting construction with a given set of distinct non-negative price tags, we get the vouchers in the order that realizes the minimum voucher cost.
\end{theorem}

\begin{proof}

Let the distinct non-negative price tags be $s_1<s_2<\cdots<s_n$.

For voucher cost, we regard the sequence as having boundary values $v_0=1$ on the left and $v_{n+1}=0$ on the right; equivalently, we are applying the Flipping Lemma to the pairwise cost of
\[
(1,v_1,\ldots,v_n,0).
\]

By Lemma~\ref{lem:pnisn}, there is a minimum-cost arrangement with $v_n=s_n$. Thus, the first step of the construction, placing the largest voucher next to the right boundary $0$, is valid.

We now proceed by induction. Suppose that after several steps, there is a minimum-cost arrangement agreeing with the construction so far. The placed
vouchers occupy the positions outside one contiguous unfilled interval. Let $\ell$ and $r$ be the positions of its left and right boundaries, with boundary values $v_\ell$ and $v_r$.

Suppose first that the next voucher prescribed by the construction is the smallest remaining voucher, say $s_k$. This means that the previous step placed the largest then-remaining voucher, so the boundary created in the previous step is larger than every voucher still unplaced. If $v_\ell<v_r$, the construction places $s_k$ at position $r-1$. Suppose
instead that, in a compatible minimum arrangement, $s_k$ is at position $t<r-1$. Apply the Flipping Lemma with $i=t-1$ and $j=r$. Then
\[
v_{i+1}-v_{j-1}=s_k-v_{r-1}<0.
\]
Also $v_i-v_j<0$: if $t=\ell+1$, this is $v_\ell-v_r<0$, and otherwise $v_i$ is still unplaced and hence is smaller than $v_r$. Thus
\[
(v_i-v_j)(v_{i+1}-v_{j-1})>0,
\]
so the flip strictly decreases the voucher cost, a contradiction. Hence $s_k$ must be placed at $r-1$. The case $v_\ell>v_r$ is symmetric.

If $v_\ell=v_r$, the same argument excludes all interior positions for $s_k$. Thus $s_k$ may be assumed to be at one of the two ends of the
unfilled interval. Reversing the entire unfilled interval does not change the
cost, since
\[
(v_\ell-v_r)(v_{\ell+1}-v_{r-1})=0.
\]
Therefore, either end is compatible with a minimum-cost arrangement.

Now suppose that the next voucher prescribed by the construction is the largest remaining voucher, say $s_k$. This means that the previous step
placed the smallest then-remaining voucher, so the boundary created in the previous step is smaller than every voucher still unplaced. If $v_\ell<v_r$, the construction places $s_k$ at position $\ell+1$.
Suppose instead that, in a compatible minimum arrangement, $s_k$ is at position $t>\ell+1$. Apply the Flipping Lemma with $i=\ell$ and $j=t+1$. Then
\[
v_{i+1}-v_{j-1}=v_{\ell+1}-s_k<0.
\]
Also $v_i-v_j<0$: if $t=r-1$, this is $v_\ell-v_r<0$, and otherwise $v_j$ is still unplaced and hence is larger than $v_\ell$. Therefore
\[
(v_i-v_j)(v_{i+1}-v_{j-1})>0,
\]
so the flip strictly decreases the voucher cost, again a contradiction. Hence $s_k$ must be placed at $\ell+1$. The case $v_\ell>v_r$ is symmetric, and the tie case $v_\ell=v_r$ is handled as above.

Thus, after every step, there exists a minimum-cost arrangement compatible with the choices made by the construction. By induction, the alternating
max-to-min and min-to-max boundary-respecting construction produces an ordering realizing the minimum voucher cost.
\end{proof}

\begin{corollary}
\label{cor:minvouchermorethan1}
If $s_1 \ge 1$, permutations of a set of vouchers realizing the minimum voucher cost are equivalent to a permutation that is order-isomorphic to
\[n-1,\ 2,\ n-3,\ 4,\ \ldots,\ 3,\ n-2,\ 1,\ n.\]
\end{corollary}

From the corollary above, we have the following corollary.

\begin{corollary}
\label{cor:mvThroughmp}
    We have
    \[m_V^S(n) = m_P^S(n) + s_{n-1}.\]
\end{corollary}

We now consider the sequence of natural numbers and calculate the minimum voucher cost $m_V(n)$.

\begin{example}
    For $n = 1$, we have $m_V(1) = 1$. For $n > 1$, we have
    \[m_V(n) = m_P(n) + n - 1 = \frac{n^3+8n}{6} - \frac{15 - 3(-1)^n}{12}.\]
The corresponding voucher sequences are
\[n-1,\ 2,\ n-3,\ 4,\ \dots,\ n-4,\ 3,\ n-2,\ 1,\ n\]
and
\[n-2,\ 3,\ n-4,\ 5,\ \dots,\ 4,\ n-3,\ 2,\ n-1,\ 1,\ n.\]
We get the following sequence, which is now sequence A393145 in the OEIS:
\[1,\ 3,\ 7,\ 15,\ 26,\ 43,\ 65,\ 95,\ 132,\ 179,\ 235,\ 303,\ 382.\]
\end{example}

\begin{corollary}
    We have
    \[m_V(n) = \textrm{A}026035(n) + n-1 = \textrm{A}110611(n-1)+n.\]
\end{corollary}

\begin{proof}
The first part follows from Corollary~\ref{cor:mvThroughmp} and the fact that 
\[m_P(n) = \textrm{A}026035(n).\]

We have $A110611(n-1)=m_L(n-1)$ by its definition. From Example~\ref{ex:mpn-mLn-1+1}, we know that
\[\textrm{A}026035(n) = m_P(n) = m_L(n-1) + 1 = \textrm{A}110611(n-1) + 1,\]
which implies the second part.
\end{proof}

We can also describe the sequence minimizing the pairwise cost using a boundary construction.

\begin{corollary}
    We can use the alternating max-to-min and min-to-max boundary-respecting construction, with initial values $0$ on both sides, to construct a sequence with minimum pairwise cost.
\end{corollary}

\begin{proof}
    We use the formula $C_P(\vec{v}) = C_P((0,\vec{v})) = C_V((0,\vec{v}))$ from Proposition~\ref{prop:add0}. The pairwise cost is equivalent to multiplying the first value by $0$ at the beginning of the voucher cost, rather than by $1$, so we initialize the left boundary to $0$ rather than $1$. From there, we can follow the boundary procedure.
\end{proof}

We can also describe the sequence minimizing the loop cost using our boundary construction.

\begin{corollary}
We can use the alternating max-to-min and min-to-max boundary-respecting construction, with initial values $s_1$ and $s_n$, and proceeding with the remaining values to construct a sequence with minimal loop cost.
\end{corollary}

\begin{proof}
We can represent our vector $\vec{v}$ as a concatenation $(s_n,\vec{w})$. From Proposition~\ref{prop:add0}, we have $C_P((s_n,\vec{w},s_n)) = C_L((s_n,\vec{w})) = C_L(\vec{v})$. If we follow the boundary procedure for the pairwise cost on vector $(s_n,\vec{w},s_n)$, we will put $s_n$ to the right, then $s_1$ to the right, and then $s_n$ to the left, and then continue. This is equivalent to initializing the boundaries as $s_1$ and $s_n$ and proceeding with the remaining values.

As a result, we will get the minimum pairwise cost for the vector $(s_n,\vec{w},s_n)$. Notice that we end up with two values of $s_n$ at both ends, implying that this cost is the same as the loop cost for $\vec{v}$, concluding the proof. Once again, rotations and reflections of the resulting sequence work as well.
\end{proof}

\begin{example}
    Suppose we have vouchers with price tags $0.5$, $1$, $4$, and $5$. We start by placing $0.5$ on the left and initializing both boundaries to $0.5$. Now, we place $5$ next to the smallest boundary, which could be either, since they are the same. If $5$ goes next to the left boundary, then we would place $1$ next to $5$. After that, we would place $4$ next to $1$. If $5$ goes next to the right boundary, then we would still place $1$ next to $5$, but $4$ would go next to $0.5$ (since $0.5<1$). This method yields two equivalent sequences $(0.5,5,1,4)$ and $(0.5,4,1,5)$, realizing the minimum. 
\end{example}

\section{Applications}
\label{sec:applications}

\subsection{More on the Voucher Cost for Natural Numbers}
\label{sec:diffandranges}

\subsubsection{Ranges}

Now we look at the difference $M_V(n) - m_V(n)$, which, for $n > 1$, equals
\[\frac{2n^3+3n^2-11n+18}{6} - \frac{n^3+8n}{6} + \frac{15-3(-1)^{n}}{12}.\]
After collecting the terms, we get
\[M_V(n) - m_V(n) = \frac{n^3+3n^2-19n}{6}+\frac{51-3(-1)^n}{12}.\]

Here are the first few terms of the difference sequence $M_V(n) - m_V(n)$ starting from index 1, which is new sequence A393532 in the OEIS \cite{OEIS}:
\[0,\ 1,\ 4,\ 10,\ 22,\ 39,\ 64,\ 96,\ 138,\ 189,\ 252,\ 326,\ 414.\]

Consider $n$ vouchers with price tags from 1 to $n$. Depending on the order, what are the possible values of the total voucher cost?

\begin{example}
    For $n=3$, we know that the minimum voucher cost of $7$ is achieved for the permutation $(2,1,3)$ and the maximum cost of $11$ is achieved for permutations $(2,3,1)$ and $(3,2,1)$. The other costs are $8$ for permutation $(3,1,2)$, $9$ for $(1,2,3)$, and $10$ for $(1,3,2)$.
\end{example}

Now, we ask a new question: Are there costs that are never achieved for $n$ price tags with values $1$ through $n$?

\begin{example}
    The maximum voucher cost for $n=1$ is $1$, and the smallest cost for $n=2$ is $3$. Thus, we can never have a cost of $2$. Similarly, the maximum cost for $n=2$ is $4$, and the smallest cost for $n=3$ is $7$, implying that $5$ and $6$ are never achieved. Similarly, $12$, $13$, and $14$ are never achieved.
\end{example}

We also know from Example~\ref{ex:cv4} that a cost of 22 is not achieved. Thus, our current list of non-achievable voucher costs is 
\[2,\ 5,\ 6,\ 12,\ 13,\ 14,\ 22.\]

We can provide examples of non-achieved costs because, for very small values of $n$, we have $M_V(n) + 1 < m_V(n+1)$. The pattern stops, as the following proposition shows.

\begin{proposition}
    We have
    \[M_V(n) > m_V(n+1)\]
    if and only if $n \ge 5$.
\end{proposition}

\begin{proof}
Expanding the formulas, we get that $M_V(n) - m_V(n+1)$ equals
\[\frac{2n^3+3n^2-11n+18}{6} - \frac{(n+1)^3+8(n+1)}{6} + \frac{15 - 3(-1)^{n+1}}{12}.\]
After simplifying, we get
\[M_V(n)-m_V(n+1) = \frac{n^3-22n}{6}+\frac{11 + (-1)^{n+1}}{4}.\]
The second fraction is non-negative. The resulting polynomial is positive whenever $n^2 > 22$, that is, when $n \ge 5$. We conclude the proof by manually checking for $1 \le n < 5$.
\end{proof}

For a given $n$, the range of voucher costs is in the order of $n^3$. On the other hand, we have $n!$ permutations. It is not surprising that our program did not find any more non-achievable costs.

\subsubsection{Average values}

We are also interested in average values. We denote them as $A_V(n)$, $A_P(n)$, and $A_L(n)$ for our cost functions $C_V$, $C_P$, and $C_L$, respectively.

We have the following values for $A_V(n)$:
\[1,\ 3\frac{1}{2},\ 9\frac{1}{3},\ 20,\ 37,\ 61\frac{5}{6},\ 96,\ 141,\ 198\frac{1}{3},\ 269\frac{1}{2},\ 356,\ 459\frac{1}{3},\ \ldots.\]

Similarly, we have the following values for $A_P(n)$:
\[0,\ 2,\ 7\frac{1}{3},\ 17\frac{1}{2},\ 34,\ 58\frac{1}{3},\ 92,\ 136\frac{1}{2},\ 193\frac{1}{3},\ 264,\ 350,\ 452\frac{5}{6}.\ \ldots.\]

Similarly, we have the following values for $A_L(n)$:
\[0,\ 4,\ 11,\ 23\frac{1}{3},\ 42\frac{1}{2},\ 70,\ 107\frac{1}{3},\ 156,\ 217\frac{1}{2},\ 293\frac{1}{3},\ 385,\ 494,\ \ldots.\]

We can get formulas for all the values.

\begin{proposition}
For $n > 1$, we have
\[
\begin{aligned}
A_L(n) &= \frac{n(n+1)(3n+2)}{12}\\
A_P(n) &= \frac{(n+1)(3n+2)(n-1)}{12}\\
A_V(n) &= \frac{(n+1)(3n^2-n+4)}{12}.
\end{aligned}
\]
\end{proposition} 

\begin{proof}
Consider two distinct integers $a$ and $b$ such that $1\leq a,b\leq n$. The product $ab$ is a term in the sum for the loop cost for $2n(n-2)!$ of the orderings of integers from 1 to $n$ since we can have $n$ places for them to be next to each other, $2$ ways to arrange the order of $a$ and $b$, and $(n-2)!$ ways to arrange the rest of the numbers. The sum of all the possible $ab$ values is
\[\frac{(1+2+\dots+n)^2-(1^2+2^2+\dots+n^2)}{2} = \frac{n^2(n+1)^2}{8} - \frac{n(n+1)(2n+1)}{12}.\]
Simplifying we get
\[
\begin{aligned}
\frac{n^2(n+1)^2}{8} - \frac{n(n+1)(2n+1)}{12}
&= n(n+1)\left(\frac{3n(n+1)}{24} - \frac{2(2n+1)}{24}\right) \\
&= \frac{n(n+1)}{24}\left(3n(n+1) - 2(2n+1)\right) \\
&= \frac{n(n+1)}{24}\left(3n^2 - n - 2\right) \\
&= \frac{n(n+1)(3n+2)(n-1)}{24}.
\end{aligned}
\]
Thus, the sum of all possible products is
\[\frac{n(n+1)(3n+2)(n-1)}{24}\cdot 2n(n-2)! = \frac{n(n+1)(3n+2)n!}{12}.\]
When we divide this by the $n!$ possible ways to arrange the integers from 1 to $n$, we get the result for the average loop cost.

Now consider the pairwise cost. Consider $n$ permutations that are rotations of each other. They all have the same loop cost, with each product participating in all sums. Each such product participates in $n-1$ sums when calculating the pairwise cost. It follows that
\[A_P(n) = \frac{n-1}{n} A_L(n) = \frac{(n+1)(3n+2)(n-1)}{12}.\]

The formula for the voucher cost is similar to the pairwise cost formula, except we must add an additional term at the beginning. The average value of the term at the beginning is $\frac{n + 1}{2}$, which when added to $\frac{(n+1)(3n+2)(n-1)}{12}$ gives us $A_V(n)=\frac{(n+1)(3n^2-n+4)}{12}$.
\end{proof}

\begin{remark*}
    Coincidentally, the formula for vouchers works in case $n=1$.
\end{remark*}

Observe that the fractional part in all three sequences follows the same pattern. In particular, it has period 12. As all of these averages are polynomials divided by 12, then $A(n+12) - A(n)$ must be an integer.

\begin{corollary}
The fractional part of each of the sequences $A_L(n)$, $A_P(n)$, and $A_V(n)$ follows the same pattern of period 12:
\[0,\ \frac{1}{3},\ \frac{1}{2},\ 0,\ \frac{1}{3},\ 0,\ \frac{1}{2},\ \frac{1}{3},\ 0,\ 0,\ \frac{5}{6},\ 0.\]
For $A_P(n)$, the pattern starts at index $2$; for $A_L(n)$, the pattern starts at index $3$; and for $A_V(n)$, the pattern starts at index $8$.
\end{corollary}

\subsection{Arithmetic progressions}
\label{sec:arithmprog}

In our original problem, the vouchers were from the set $\{1, \ldots, n\}$. What if the voucher price tags are integers forming an arbitrary arithmetic progression? Suppose we replace voucher $k$ with $a+kd$.

From the Transcalation Lemma~\ref{lem:scaling}, we get the following formula
\[C_L(\vec{v}) = d^2C_L(\pi) + adn(n+1) + a^2n.\]

We denote the maximum and minimum as $M_L(n,a,d)$ and $m_L(n,a,d)$, respectively. We give similar notation to pairwise and voucher costs.

Using the order-isomorphism, and plugging in $M_L(n)$ in place of $C_L(\pi)$ for maximum and $m_L(n)$ in place of $C_L(\pi)$ for minimum, we get the following formulas
\[M_L(n,a,d)=d^2\left(\frac{2n^3+3n^2-11n+18}{6}\right)+adn(n+1)+a^2n,\]
\[m_L(n,a,d)=d^2\left(\frac{n^3+3n^2+5n}{6} - \frac{9 + 3(-1)^n}{12}\right)+adn(n+1)+a^2n.\]

Throughout this section, we consider the sequence of odd numbers as our main example: $v_i = 2i-1$, corresponding to $a = -1$ and $d = 2$.

Recall that for the formulas above, we assume that $n > 1$.

\begin{example}
    For odd numbers, we get
    \[M_L(n,-1,2) = 4M_L(n) -2n(n+1) + n = \frac{4n^3-25n+36}{3},\]
    matching sequence A273314 in the OEIS, up to a shift in index. The first few terms starting from index $2$ are
    \[6,\ 23,\ 64,\ 137,\ 250,\ 411,\ 628,\ 909,\ 1262.\]
    Similarly, for the minimum loop cost
    \[m_L(n,-1,2) = 4m_L(n) - 2n(n+1) + n= \frac{2n^3+7n-9-3(-1)^n}{3}.\]
    The corresponding sequence starts from index $2$ as 
\[6,\ 23,\ 48,\ 93,\ 154,\ 243,\ 356,\ 505,\ 686,\]
and is now sequence A395196 in the OEIS database \cite{OEIS}.
\end{example}

For the pairwise cost, we have
\[C_P(\vec{v}) = d^2C_L(\pi) + adn(n+1) + a^2n - v_1v_n.\]

The maximum pairwise cost is achieved in the same order as the maximum loop cost. We have $v_1 = s_2 = a+2d$ and $v_n = s_1 = a+d$; plugging this in, we get for $n > 1$
\[M_P(n,a,d) = M_L(n,a,d) - (a+2d)(a+d).\]

For the minimum pairwise cost, the calculation is slightly different. For order $\vec{v}$ that realizes the minimum pairwise cost, we have $v_1 = a + d(n-1)$ and $v_n = a +dn$. It follows that 
\[C_L(\pi) = C_P(\pi) + n(n-1).\]
Thus, 
\[C_P(\vec{v}) = d^2(C_P(\pi) + n(n-1)) + adn(n+1) + a^2n - (a + d(n-1))(a+dn).\]
Simplifying, we get for $n > 1$
\[m_P(n,a,d) = d^{2}m_{P}(n)+ a^{2}(n-1) + a d(n^{2}-n+1).\]

\begin{example}
    For odd numbers and $n > 1$, we get
    \[M_P(n,-1,2) = M_L(n,-1,2) - 3
 = \frac{4n^3-25n+27}{3},\]
 which equals sequence A273314 in the OEIS with $3$ subtracted. The first few maximum pairwise costs starting from index $2$ are
 \[3,\ 20,\ 61,\ 134,\ 247,\ 408,\ 625,\ 906,\ 1259.\]
    Similarly, for $n > 1$
    \[
\begin{aligned}
m_P(n,-1,2)
&= 4m_P(n) + n - 1 - 2\bigl(n^2 - n + 1\bigr) \\
&= 4m_P(n) - 2n^2 + 3n - 3
 = \frac{2n^3 - 6n^2 + 13n - 12 + 3(-1)^n}{3}.
\end{aligned}
\]
The corresponding sequence, the minimum pairwise cost for odd vouchers $m_P(n,-1,2)$, starts from index $2$ as 
\[3,\ 8,\ 25,\ 50,\ 95,\ 156,\ 245,\ 358,\ 507,\]
and is now sequence A394241 in the OEIS database \cite{OEIS}.
\end{example}

We now move to the voucher cost, which is
\[C_V(\vec{v}) = v_1 + d^2C_L(\pi) + adn(n+1) + na^2 - v_1v_n.\]

From Theorem~\ref{thm:maxorderiso}, we know that the maximal voucher cost is achieved when the permutation is order-isomorphic to 
\[2,\ 4,\ 6,\ 8,\ \dots,\ n,\ \dots,\ 7,\ 5,\ 3,\ 1.\]
Thus, $v_1=a + 2d$ and $v_n=a + d$. As above, the maximal voucher cost is
\[M_V(n,a,d) = M_L(n,a,d) - (a+2d)(a+d) + a + 2d.\]

Plugging in the value of $M_L(n)$, we get
\[M_V(n,a,d) = a+2d+(n-1)a^2+(n(n+1)-3)ad+\left(\frac{2n^3+3n^2-11n+6}{6}\right)d^2.\]

\begin{example} The maximum voucher cost for odd numbers is $\frac{4n^3-25n+36}{3}$ for $n > 1$, and 1 for $n=1$. This is the same as the loop cost. This appears to be sequence A273314 in the OEIS, shifted by 1 in index. The actual sequence for the voucher cost starts as
\[1,\ 6,\ 23,\ 64,\ 137,\ 250,\ 411,\ 628,\ 909,\ 1262,\ 1695.\]
\end{example}

We now move to the minimum. From Corollary~\ref{cor:minvouchermorethan1}, we know that the minimum voucher cost is achieved when the permutation is order-isomorphic to 
\[n-1,\ 2,\ n-3,\ 4,\ \ldots,\ 3,\ n-2,\ 1,\ n.\]
Thus, $v_1=a + (n-1)d$ and $v_n=a + nd$, so the minimum voucher cost is 
\[d^2(m_L(n)-n(n-1)) + ad(n(n+1)-2n+1) + a^2(n-1) + a + d(n-1).\]

\begin{example}
The minimum voucher cost for odd numbers with $n$ vouchers is $\frac{2n^3-6n^2+19n-21+3(-1)^n}{3}$ for $n > 1$, and 1 for $n=1$. The sequence is now sequence A392324 in the OEIS \cite{OEIS} and starts as: 1, 4, 11, 30, 57, 104, 167, 258, 373, 524, 707.
\end{example}

\subsection{Geometric Progression}
\label{sec:geometricprog}

We consider the set of vouchers that are powers of $x$. We assume that $x > 1$. 

For our formulas, we start with the voucher $x^0=1$. For our functions, we add $x$ as a parameter, that is, our functions are $M_L^G(n,x)$, $M_P^G(n,x)$, $M_V^G(n,x)$, $m_L^G(n,x)$, $m_P^G(n,x)$, and $m_V^G(n,x)$.

We assume that $n$ is $2$ or greater. By direct computation, we get the following formulas.

\begin{theorem}
    For $n \ge 2$, we get the following formulas for maximum:
    \[M_V^G(n,x) = M_L^G(n,x) = M_P^G(n,x) + x = \frac{x^2-x^{2n-2}}{1-x^2} + x^{2n-3} + x,\]
    and the following for minimum
    \[m_V^G(n,x) = m_P^G(n,x) + x^{n-2},\]
    where
    \[m_V^G(n,x) = \frac{nx^{n-3}(x^2+1)}{2} - \left(\frac{1-(-1)^n}{2}\right)\frac{x^{n-3}(x-1)^2}{2} + x^{n-2} - x^{n-3}.\]
    We also have
    \[m_L^G(n,x) = \frac{nx^{n-2}(x^2+1)}{2} - \left(\frac{3+(-1)^n}{4}\right)x^{n-2}(x-1)^2.\]
\end{theorem}

\begin{example}
Suppose $x = 2$. Then the maximum voucher and loop costs are the following sequence:
\[1,\ 4,\ 14,\ 54,\ 214,\ 854,\ 3414,\ 13654,\ 54614,\ 218454,\ 873814.\]
For $n> 1$, the formula is
\[M_V^G(n,2) = \frac{5\cdot 4^{n-1} + 4}{6}.\]
For $n > 1$, this sequence is A080675$(n-1) + 2$, while the pairwise cost is sequence A080675 in the database \cite{OEIS}.

The minimum voucher cost is the following sequence, which is now sequence A395767 in the OEIS \cite{OEIS}:
\[1,\ 3,\ 8,\ 22,\ 52,\ 128,\ 288,\ 672,\ 1472,\ 3328,\ 7168,\ 15872,\ \ldots.\]

The formula for $n > 1$ is 
\[m_V^G(n,2) = 2^{n-5}(10n + 3 + (-1)^n).\]

The minimum pairwise cost is
\[m_P^G(n,2) = 2^{n-5}(10n - 5 + (-1)^n) = m_V(n,2) - 2^{n-2},\]
and the first few terms starting from index $2$ are
\[2,\ 6,\ 18,\ 44,\ 112,\ 256,\ 608,\ 1344,\ 3072,\ 6656,\ 14848,\ 31744,\ 69632,\]
which is now sequence A395770 in the OEIS \cite{OEIS}.

The minimum loop cost for $n > 1$ is
\[2^{n-4}(10n - 3 - (-1)^n),\]
and the first few terms starting from index $2$ are now sequence A394572 in the OEIS \cite{OEIS}:
\[4,\ 14,\ 36,\ 96,\ 224,\ 544,\ 1216,\ 2816,\ 6144,\ 13824,\ 29696,\ 65536,\ 139264.\]

\end{example}

\subsection{Negative Numbers}
\label{sec:neg}

Suppose the price tags have negative values. We can look at these values as $-\vec{v}$, with $\vec{v}$ having positive coordinates. Then
\[C_L(-\vec{v}) = C_L(\vec{v}) \quad \textrm{ and } \quad C_P(-\vec{v}) = C_P(\vec{v}).\]
This connection allows us to find the maximum and minimum loop and pairwise costs for a set of vouchers with negative price tags.

Now we consider the voucher cost
\[C_V(-\vec{v}) = - v_1 + v_1v_2 + v_2v_3 + \dots + v_{n-1}v_n.\]
It is very similar to the voucher cost, except the cost of the first voucher changes. We will discuss what happens here in a more general setting, introducing a parameter $c$. Consider a new voucher cost depending on $c$
\[C_V^c(-\vec{v}) = cv_1 + v_1v_2 + v_2v_3 + \dots + v_{n-1}v_n.\]

\begin{proposition}
    The min-to-min boundary-respecting construction, initialized with left boundary $c$ and right boundary $0$, produces an ordering that maximizes the generalized voucher cost
    \[
    C_V^c(-\vec{v}) = cv_1 + v_1v_2 + v_2v_3 + \cdots + v_{n-1}v_n
    \]
    for any real number $c$ and any set of non-negative price tags.
\end{proposition}

\begin{proof}
Let the non-negative price tags be $s_1\le s_2\le\cdots\le s_n$. We regard the cost
$C_V^c(-\vec{v})$ as the pairwise cost of the augmented sequence $(c,v_1,\ldots,v_n,0)$. Thus, the initial boundary values are $c$ on the left and $0$ on the right.

We prove by induction that after each step of the min-to-min boundary-respecting construction, there is a maximum-cost arrangement agreeing with the construction so far.

Suppose that the vouchers $s_1,\ldots, s_{k-1}$ have already been placed according to the construction. The placed vouchers occupy the positions outside one contiguous unfilled interval. Let the left and right boundary values of this interval be $v_\ell$ and $v_r$, respectively. We need to show that the next voucher, $s_k$, may be placed next to a smallest boundary.

Assume first that $v_\ell<v_r$. The construction places $s_k$ at position $\ell+1$. Suppose that in a maximum-cost arrangement compatible with the previous choices, $s_k$ is instead at position $j-1$, where $j-1>\ell+1$. Then $v_{\ell+1}\ge s_k$, because $s_k$ is one of the smallest unplaced vouchers.

We apply the Flipping Lemma~\ref{lem:flipping} to the subsequence from $v_{\ell+1}$ through $v_{j-1}$. The difference between the old cost and the new cost is
\[
(v_\ell-v_j)(v_{\ell+1}-v_{j-1}).
\]
Here $v_{\ell+1}-v_{j-1}=v_{\ell+1}-s_k\ge 0$. Also, $v_\ell-v_j<0$: if $j=r$, this follows from $v_\ell<v_r$, and otherwise $v_j$ is still unplaced and is at least $s_k\ge v_\ell$; if equality occurs, the product is $0$, and the flip does not change the cost. Therefore,
\[
(v_\ell-v_j)(v_{\ell+1}-v_{j-1})\le 0.
\]
By the Flipping Lemma~\ref{lem:flipping}, the flip does not decrease the cost. Thus, there is a maximum-cost arrangement in which $s_k$ is placed next to the left boundary.

The case $v_r<v_\ell$ is symmetric.

If $v_\ell=v_r$, the same argument shows that $s_k$ may be assumed to be at one of the two ends of the unfilled interval. Since the two boundary values are equal, reversing the unfilled interval does not change the cost. Hence, either end is compatible with a maximum-cost arrangement.

Thus, at each step, the next smallest remaining voucher may be placed next to a smallest boundary. By induction, the min-to-min boundary-respecting construction produces an ordering realizing the maximum value of $C_V^c(-\vec{v})$.
\end{proof}

\begin{remark*}
    We usually use distinct price tags, but in the theorems above and below, we allow the price tags to repeat, because one of them could equal $c$.
\end{remark*}

\begin{proposition}
    The alternating max-to-min and min-to-max boundary-respecting construction, initialized with left boundary $c$ and right boundary $0$, produces an ordering that minimizes the generalized voucher cost
    \[
    C_V^c(-\vec{v}) = cv_1 + v_1v_2 + v_2v_3 + \cdots + v_{n-1}v_n
    \]
    for any real number $c$ and any set of distinct non-negative price tags.
\end{proposition}

\begin{proof}
Let the price tags be $s_1 \le s_2 \le \cdots \le s_n$. We regard the cost
$C_V^c(-\vec{v})$ as the pairwise cost of the augmented sequence
\[
(c,v_1,\ldots,v_n,0).
\]
Thus, the initial boundary values are $c$ on the left and $0$ on the right.

We prove by induction that after each step of the construction, there is a
minimum-cost arrangement agreeing with the construction so far. The placed
vouchers occupy the positions outside one contiguous unfilled interval. Let
$v_\ell$ and $v_r$ be the left and right boundary values of this interval.

Suppose first that the next voucher prescribed by the construction is the
largest remaining voucher, say $s_k$. This means that the smaller current
boundary is less than or equal to every voucher still unplaced. If
$v_\ell<v_r$, the construction places $s_k$ at position $\ell+1$. Suppose
instead that, in a compatible minimum arrangement, $s_k$ is at position
$t>\ell+1$. Apply the Flipping Lemma~\ref{lem:flipping} with $i=\ell$ and $j=t+1$. Then
\[
v_{i+1}-v_{j-1}=v_{\ell+1}-s_k\le 0.
\]
Also $v_i-v_j\le 0$: if $t=r-1$, this is $v_\ell-v_r<0$, and otherwise
$v_j$ is still unplaced and hence is at least $v_\ell$. Therefore,
\[
(v_i-v_j)(v_{i+1}-v_{j-1})\ge 0.
\]
By the Flipping Lemma~\ref{lem:flipping}, the flip does not increase the cost. Hence there is a minimum-cost arrangement in which $s_k$ is placed at position $\ell+1$. The case $v_r<v_\ell$ is symmetric.

Now suppose that the next voucher prescribed by the construction is the smallest remaining voucher, say $s_k$. This means that the larger current boundary is greater than or equal to every voucher still unplaced, as we already placed the largest voucher next to a boundary. If $v_\ell<v_r$, the construction places $s_k$ at position $r-1$. Suppose instead that, in a compatible minimum arrangement, $s_k$ is at position $t<r-1$. Apply the Flipping Lemma~\ref{lem:flipping} with $i=t-1$ and $j=r$. Then
\[
v_{i+1}-v_{j-1}=s_k-v_{r-1}\le 0.
\]
Also $v_i-v_j\le 0$: if $t=\ell+1$, this is $v_\ell-v_r<0$, and otherwise
$v_i$ is still unplaced and hence is at most $v_r$. Therefore,
\[
(v_i-v_j)(v_{i+1}-v_{j-1})\ge 0.
\]
By the Flipping Lemma~\ref{lem:flipping}, the flip does not increase the cost. Hence there is a minimum-cost arrangement in which $s_k$ is placed at position $r-1$. The case $v_r<v_\ell$ is symmetric.

If $v_\ell=v_r$, the same argument shows that the prescribed voucher may be assumed to be at one of the two ends of the unfilled interval. Since the boundary values are equal, reversing the entire unfilled interval does not change the cost.

Thus, after every step, there exists a minimum-cost arrangement compatible with the construction. By induction, the alternating max-to-min and min-to-max boundary-respecting construction produces an ordering realizing the minimum value of $C_V^c(-\vec{v})$.
\end{proof}

\begin{example}
    Suppose the vector of price tags is $(-1,-2,-3,-4,-5)$. The order $(-1,-3,-5,-4,-2)$ realizes the maximum voucher cost of $45$. The order $(-5,-1,-3,-2,-4)$ realizes the minimum voucher cost of $17$.
\end{example}

\section{Acknowledgments}

We are grateful to the PRIMES STEP program for giving us the opportunity to conduct this research.

\end{document}